\documentclass{amsart}
\usepackage{graphicx}
\usepackage{fullpage}
\usepackage{url}

\usepackage{tikz}
\tikzset{partition/.style={fill,circle,inner sep=1pt}}
\usepackage{tikz}
\usetikzlibrary{shapes,arrows}
\tikzstyle{pnt}=[draw,ellipse,fill,inner sep=1pt]
\tikzstyle{opnt}=[draw,ellipse,inner sep=1pt]
\tikzstyle{pntt}=[draw,ellipse,fill,inner sep=0.5pt]
\tikzstyle{point}=[draw,ellipse,fill,inner sep=2pt]
\usepackage{color}

\newtheorem{proposition}{Proposition}

\newtheorem{definition}{Definition}
\newtheorem{example}{Example}

\title{Constructing Skolem sequences via generating trees}
\author{Sophie Burrill, Lily Yen}
\address[S.~Burrill, L.~Yen]{Department of Mathematics, Simon Fraser University, Burnaby, BC, Canada}
\email{srb7@sfu.ca}
\address[L.~Yen]{Department of Mathematics and Statistics, Capilano University, North Vancouver, BC, Canada} 

\begin{document}

\begin{abstract}
A Skolem sequence is a linear arrangement of the multiset, $\{1, 1, 2, 2, \dots, n, n\}$ such that if $r \in [n]$ appears in positions $i$ and $j$, then $|i-j| = r$. We first translate the problem to a particular set of perfect matchings, then apply the method of generating trees for open arc diagrams to generate exhaustively all Skolem sequences of a given size. Tracking  the arc length between pairs of vertices in an arc annotated diagram is the central task. Although we do not surpass previously known enumerative results, this method drastically reduces the search space compared to previously known methods. 
\end{abstract}
\maketitle

\section{Introduction}
\label{sec:in}

\emph{Arc diagrams}, also known as arc annotated sequences, are structures that encode a variety of combinatorial classes, including matchings \cite{ChDeDuStYa2007}, colored matchings \cite{ChG2011}, set partitions \cite{ChDeDuStYa2007}, permutations \cite{BuMiPo2010}, labelled graphs \cite{deMier2007}, tangled diagrams \cite{ChQiRe2008} and RNA substructures \cite{JiRe2010}. Each arc diagram is a row of increasingly labelled vertices from $1$ to $n$ with some arcs between them, restricted according to the combinatorial class being represented. Much recent work has focused on illustrating the equidistribution of two statistics that arise in arc diagrams: $k$-crossings (a set of $k$ mutually crossing arcs) and $k$-nestings (a set of $k$ mutually nesting arcs)  (\cite{Kr2006}, \cite{ChDeDuStYa2007},  \cite{BuMiPo2010}, \cite{ChQiRe2008}, \cite{deMier2007}). Also of interest has been the enumeration of combinatorial classes parameterized by these statistics (\cite{ChDeDuStYa2007}, \cite{BoXi2007}, \cite{MiYe2011}).  More recently, in \cite{BuElMiYe2011}, a generalized version of an arc diagram was used  for the first time to construct generating trees and functional equations for $k$-nonnesting set partitions and permutations. This generalization was the \emph{open arc diagram}, which allowed for arcs to have left endpoints but no right endpoints, and was introduced for bijective purposes in \cite{KaZe06}.

Here we represent a different combinatorial class,  \emph{Skolem sequences}, with arc diagrams, and utilize open arc diagrams to build a generating tree that exhaustively constructs all Skolem sequences of order $n$ using a strategy similar to that seen in \cite{BuElMiYe2011}. 
 While this direct generation scheme is restricted in practical effectiveness due to memory limitations, we feel this method is valuable for three reasons:
\begin{enumerate}
\item This is a new approach that \emph{recursively} constructs Skolem sequences;
\item The strategy is a systematic approach for various generalizations of Skolem sequences; 
\item This method fits within the larger picture of open arc diagrams being used as a tool for generation and enumeration. 
\end{enumerate}  

\section{Plan of the paper} \label{sec:plan}

The paper will proceed as follows: in Section \ref{sec:Sk} Skolem sequences will be defined and background from the literature will be put into context. In Section \ref{sec:ska} we will depict Skolem sequences as arc diagrams, and introduce a generalization of them called \emph{open Skolem sequences} which will be critical to our generation scheme. In Section \ref{sec:gentree} we present a label for open Skolem sequences that allows us to formalize a succession rule for generating all descendants of an open Skolem sequences. We finish with a discussion in Section \ref{sec:dis} which places memory limitations into context, highlights our contribution to the study of Skolem sequences and potential avenues for other generalizations of Skolem like sequences. Lastly Skolem sequences are viewed as a subclass of perfect matchings and an observation is made regarding their corresponding oscillating tableaux as seen in \cite{ChDeDuStYa2007}.

\section{Skolem sequences} \label{sec:Sk}

Skolem sequences were introduced in 1957 \cite{Sk1957}, when Thoralf Skolem asked about the possibility of partitioning the set $\{1, 2 \ldots, 2n\}$ into $n$ pairs $(a_r, b_r)$ where $b_r-a_r=r$ for $r=1,2,  \ldots, n$.  Such a partition forms what is now referred to as a Skolem sequence. 

\begin{definition}\label{def:skolem}
A \textbf{Skolem sequence} of order $n$ is an integer sequence $w=(s_1, s_2, \ldots, s_{2n})$ of the multiset $\{1,1, 2,2, \ldots, n, n\}$ of size $2n$  satisfying the following conditions: 
\begin{enumerate}
\item For all $k\in\{1,2, \ldots n\}$, there exists exactly two elements $s_i, s_j$ in $w$ such that $s_i=s_j=k$. 
\item If $s_i=s_j=k$, and $i<j$, then $j-i=k$. 
\end{enumerate}
\end{definition}
We let $\mathcal{S}$ denote the set of all Skolem sequences and $\mathcal{S}_n$ to be of order $n$. 

\begin{example} When $n=4$, $w= (3,4,2,3,2,4,1,1) \in \mathcal{S}_4$ is a Skolem sequence. 
\end{example}

Skolem sequences have immediate further applications in design theory, particularly, in \emph{Steiner Triple Systems}.  Let $STS(n)$ denote the set of pairs (V,B) where $|V|=v$ and $B$ consists of $3$-subsets (`blocks') of $V$ such that any 2-subset of $V$ is included in exactly one block. These are the Steiner Triple Systems of size $n$. In 1958, Skolem \cite{Sk1958} showed how to build an element of $STS(6n+1)$ given an element of $\mathcal{S}_n$. Specifically, if $w=(s_1, \ldots s_{2n}) \in \mathcal{S}_n$,  an $STS(6n+1)$ can be constructed with base blocks of the form $(x, x+k,x+j+n )$, where $s_i=s_j=k$, $i<j$ and $x\in \{0,1,2 \ldots, 6n\}$.

\begin{example} Let $n=4$. The Skolem sequence $w=(3,4,2,3,2,4,1,1)\in \mathcal{S}_4$ can be written as: 

\[
   \begin{array}{*{9}{c}}
      i={}& 1& 2& 3&4&5&6&7&8\\
 s_i={}& 3&4&2&3&2&4&1&1\\
   \end{array}
\]
We build base blocks of a $STS(25)$ from this example using $x=0$. Notice that $s_1=s_4=3$, so our first base block is $(0, 0+3,0+ 4+4)=(0,3,8)$. Similarly, $s_2=s_6=4$ gives the base block $(0, 0+4, 0+6+4)=(0,4, 10)$; $s_3=s_5=2$ gives $(0, 0+2, 0+5+4)=(0,2,9)$ and $s_7=s_8=1$ gives $(0, 0+1, 0+8+4)=(0,1,12)$. Thus our base blocks for creating the $STS(25)$ are
\[(0,3, 8), (0, 4,10), (0, 2,9), \mbox{ and } (0,1,12).\]
Adding integers $1,2, \ldots, 24$ to each element of each base block and reducing \emph{mod} $25$ constructs an $STS(25)$. 

\end{example}

Beyond STSs, Skolem sequences can also be applied to the areas of starters \cite{PiSh2000}, balanced tertiary designs \cite{Bi1987}, factorizations of complete graphs \cite{CoRo1990} and labellings of graphs \cite{Ab1991}. For this reason, generation of Skolem sequences of arbitrary order is valuable, although most existing algorithms only construct a single, or at least limited number, of Skolem sequences of order $n$. Despite this, since its introduction, most work has centered around the existence of Skolem and Skolem-type sequences, the generation of individual Skolem sequences of large order, and the strict enumeration of Skolem sequences. 

Skolem proved in 1957 that Skolem sequences only exist when $n\equiv 0,1 \mod 4$ \cite{Sk1957}. 
There are many methods of generalizing Skolem sequences. These include adding a \emph{hook}, or 0, into one or multiple positions in the sequence, allowing $\lambda>2$ copies of an integer, or stipulating that each integer $i\in \{1, \ldots, n\}$ must be present $\lambda_i$ times in the sequence. In \cite{FrMe2009}, Navena Franceti\'{c} and Eric Mendelsohn gave a thorough survey of the known existence results regarding many of these generalizations.

Eldin~\emph{et al.} \cite{ElShAl1998} gave a hill climbing algorithm to generate Skolem sequences of arbitrary order, easily constructing Skolem sequences of order 84 in examples.

The enumeration of  Skolem sequences has proven quite challenging. In 1986, Jarom\'{i}r Abrham determined a lower bound for the number of Skolem sequences using a construction called  additive permutations. He showed that if $n\equiv 0,1 \mod 4$, then $|\mathcal{S}_n| \geq 2^{\lfloor \frac{n}{3} \rfloor}$ \cite{Ab1986}. Other than such bounds, exact enumeration of the number of Skolem sequences has seen most success with inclusion-exclusion algorithms \cite{GoWebsite}, \cite{LaThesis}.

The exact number of Skolem sequences of order $n$ has been computed  for $n\leq 23$ only (OEIS A004075), a limited number that is mainly due to the very large search space. Before the inclusion-exclusion algorithms were introduced \cite{GoWebsite} \cite{LaThesis}, enumerative results were restricted to the exhaustive generation of all $(2n)!$ permutations of $\{1,2, \ldots, 2n\}$, and counting the valid ones.

Strategies for exhaustive generation of Skolem sequences have also been quite limited. While \cite{GoWebsite} and \cite{LaThesis} used an inclusion-exclusion algorithm to enumerate Skolem sequences of order up to 23, these methods did not create the sequences, only counted them. It has been classically believed that the only way of generating all possible Skolem sequences of order $n$ is to check all $(2n)!$ permutations. 
 Since, as stated, Skolem sequences can be applied to many other areas, the actual construction of all Skolem sequences is very useful. The strategy we propose here can theoretically 
generate all Skolem sequences up to computational limitations. In our 
computing environment, we were able to actually generate all of the 
sequences up to size 7, and some of size 8. The main characteristic of this strategy is a significantly reduced search space.




\section{Skolem sequences as arc diagrams}
\label{sec:ska}

As suggested by the original problem posed by Skolem, the \emph{partitioning} of the set $\{1,2, \ldots, 2n\}$ provides a natural way to depict these sequences using arc diagrams. Vertices are labeled from $1$ to $2n$ of the $2n$ positions in the Skolem sequence and if entries $s_i=s_j$ in $w \in \mathcal{S}_n$, then an arc $(i,j)$ is drawn between vertices $i$ and $j$ in the set partition arc diagram. 
In an arc diagram we define the \emph{length} of a given arc $(i, j)$, where $|j-i|=k$, as $k$. Notice that the length of an arc in a Skolem sequence's arc diagram is $s_i=s_j=k$. 
\begin{example} The Skolem sequence $w=(3,4,2,3,2,4,1,1) \in \mathcal{S}_4$ is depicted as the following arc diagram: 

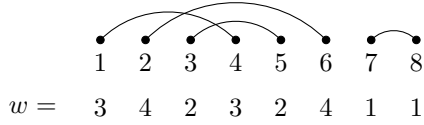
\begin{center}
\begin{tikzpicture}[scale=0.6]
\node[pnt, label=below:$1$] at (0,0){};
\node[pnt,label=below:$2$] at (1,0){};
\node[pnt, label=below:$3$] at (2,0){};
\node[pnt,label=below:$4$] at (3,0){};
\node[pnt, label=below:$5$] at (4,0){};
\node[pnt,label=below:$6$] at (5,0){};
\node[pnt, label=below:$7$] at (6,0){};
\node[pnt,label=below:$8$] at (7,0){};
\node at (0, -1.5){$3$};
\node at (1, -1.5){$4$};
\node at (2, -1.5){$2$};
\node at (3, -1.5){$3$};
\node at (4, -1.5){$2$};
\node at (5, -1.5){$4$};
\node at (6, -1.5){$1$};
\node at (7, -1.5){$1$};
\node at (-1.5, -1.5){$w=$};
\draw[bend left=45](0,0) to (3,0);
\draw[bend left=45](1,0) to (5,0);
\draw[bend left=45](2,0) to (4,0);
\draw[bend left=45](6,0) to (7,0);
\end{tikzpicture}
\end{center}

\end{example}

Once a Skolem sequence is represented this way, we can easily notice that it is not only a special case of a set partition, but also a special case of a perfect matching. In \cite{BuElMiYe2011} a generating tree strategy was employed to enumerate $k$-nonnesting set partitions using a carefully chosen label; here a similar method will be employed using \emph{open} Skolem sequences. However, instead of a label that keeps track of a nesting index, we will track arc lengths.

\begin{definition} An \textbf{open Skolem sequence} of order $n$, $\rho=(s_1, \ldots, s_n)$  is a decorated integer sequence  made up of integers $\leq n$, some possibly decorated with a $*$, such that the following conditions hold: 
\begin{enumerate}
\item If $s_i$ and $s_j$ are not decorated with a $*$, then there are exactly two elements $s_i$ and $s_j$ such that $s_i=s_j=k$, and $|j-i|=k$. In the arc diagram, an arc of length $k$ is drawn between vertices in positions $i$ and $j$. Also, such a $k$ is unique. 
\item If $s_i$  and $s_j$ are both decorated with a $*$, then $s_i\neq s_j$ for $j\neq i$. If $s_i=*k$, then in the arc diagram, vertex $i=(n+1)-k$ is an open arc that is not closed. Notice that $i$ is the minimum length that the open arc must be when completed. 
\end{enumerate}
\end{definition}
We denote the set of all open Skolem sequences as $\mathcal{OS}$ and those of order $n$ with $\mathcal{OS}_n$. Notice that for $\rho \in \mathcal{OS}_n$,  $|\rho|=n$, while for $w \in \mathcal{S}_n$,  $|w|=2n$, so the corresponding arc diagrams of a open Skolem sequence and Skolem sequence, each of order $n$ are $n$ and $2n$ respectively. 

To use the language of \cite{BuElMiYe2011}, open Skolem sequences can be represented with open arc diagrams, where the correspondence between elements in the open Skolem sequence and vertices in the diagram is as follows:

\begin{center}
\begin{tabular}{l|c|c}
\textbf{Vertex type} & \textbf{Arc Diagram} & \textbf{$s_i=$}\\
\hline
semi-opener & \begin{tikzpicture}[scale=0.7]
\node[pnt, label=below:$i$] at (0,0){};
\draw[bend left=45](0,0) to (0.5, 0.5);
\end{tikzpicture} & $*k$\\
opener & \begin{tikzpicture}[scale=0.7]
\node[pnt, label=below: $i$] at (0,0){};
\node[pnt, label=below:$j$] at (1,0){};
\draw[bend left=45](0,0) to (0.5, 0.5);
\draw[bend left=45, dotted](0.5, 0.5) to (1,0);
\end{tikzpicture} & $k$\\
closer & \begin{tikzpicture}[scale=0.7]
\node[pnt, label=below: $j$] at (0,0){};
\node[pnt, label=below:$i$] at (1,0){};
\draw[bend left=45,dotted](0,0) to (0.5, 0.5);
\draw[bend left=45](0.5, 0.5) to (1,0);
\end{tikzpicture} & $k$\\

\end{tabular}
\end{center}

\begin{example} \label{ex:iss7}Consider the open Skolem sequence $\rho=(*7, 4,1,1,*3, 4,*1)\in \mathcal{OS}_7$ and its arc diagram representation: 

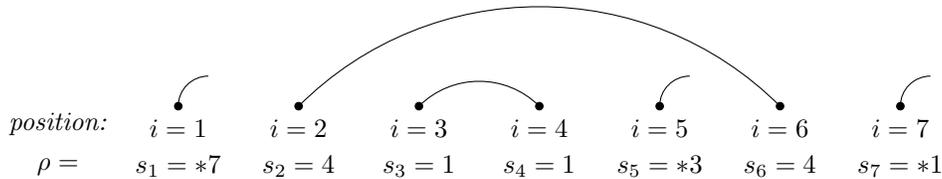
\begin{center}
\begin{tikzpicture}[scale=0.8]
\node[pnt, label=below:{$i=1$}] at (0,0){};
\node[pnt,label=below:{$i=2$}] at (2,0){};
\node[pnt, label=below:{$i=3$}] at (4,0){};
\node[pnt,label=below:{$i=4$}] at (6,0){};
\node[pnt, label=below:{$i=5$}] at (8,0){};
\node[pnt,label=below:{$i=6$}] at (10,0){};
\node[pnt, label=below:{$i=7$}] at (12,0){};
\draw[bend left=45](0,0) to (0.5, 0.5);
\draw[bend left=45](2,0) to (10,0);
\draw[bend left=45](4,0) to (6,0);
\draw[bend left=45](8,0) to (8.5, 0.5);
\draw[bend left=45](12,0) to (12.5, 0.5);
\node at (0, -1.0){$s_1=*7$};
\node at (2, -1.0){$s_2=4$};
\node at (4, -1.0){$s_3=1$};
\node at (6, -1.0){$s_4=1$};
\node at (8, -1.0){$s_5=*3$};
\node at (10, -1.0){$s_6=4$};
\node at (12, -1.0){$s_7=*1$};
\node at (-2, -1.0){$\rho=$};
\node at (-2, -0.3){position:};

\end{tikzpicture}
\end{center}

\end{example}

From an open Skolem sequence of size $n$ we can construct its descendants through the addition of these two distinct vertex types to its arc diagram:

\begin{enumerate}
\item An opener (see vertices in position $1,2,3,5$ and $7$ in Example \ref{ex:iss7}) may always be added in position $n+1$.
\item A closer, (see vertices in positions $4$ and $6$ in Example \ref{ex:iss7}) may be added in position $n+1$ provided an available semi-opener exists. 
\end{enumerate}

\begin{example} \label{ex:oss7} Consider two descendants of the open Skolem sequence $w\in \mathcal{OS}_7$ seen in Example \ref{ex:iss7} that will each have order 8:
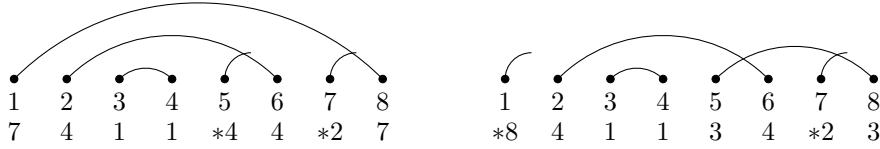
\begin{center}
\begin{tikzpicture}[scale=0.7]
\node[pnt, label=below:{$1$}] at (0,0){};
\node[pnt,label=below:{$2$}] at (1,0){};
\node[pnt, label=below:{$3$}] at (2,0){};
\node[pnt,label=below:{$4$}] at (3,0){};
\node[pnt, label=below:{$5$}] at (4,0){};
\node[pnt,label=below:{$6$}] at (5,0){};
\node[pnt, label=below:{$7$}] at (6,0){};
\node[pnt, label=below:{$8$}] at (7,0){};
\draw[bend left=45](0,0) to (7,0);
\draw[bend left=45](1,0) to (5,0);
\draw[bend left=45](2,0) to (3,0);
\draw[bend left=45](4,0) to (4.5, 0.5);
\draw[bend left=45](6,0) to (6.5, 0.5);
\node at (0, -1.0){$7$};
\node at (1, -1.0){$4$};
\node at (2, -1.0){$1$};
\node at (3, -1.0){$1$};
\node at (4, -1.0){$*4$};
\node at (5, -1.0){$4$};
\node at (6, -1.0){$*2$};
\node at (7, -1.0){$7$};

\end{tikzpicture}\hspace{1cm}
\begin{tikzpicture}[scale=0.7]
\node[pnt, label=below:{$1$}] at (0,0){};
\node[pnt,label=below:{$2$}] at (1,0){};
\node[pnt, label=below:{$3$}] at (2,0){};
\node[pnt,label=below:{$4$}] at (3,0){};
\node[pnt, label=below:{$5$}] at (4,0){};
\node[pnt,label=below:{$6$}] at (5,0){};
\node[pnt, label=below:{$7$}] at (6,0){};
\node[pnt, label=below:{$8$}] at (7,0){};
\draw[bend left=45](0,0) to (0.5,0.5);
\draw[bend left=45](1,0) to (5,0);
\draw[bend left=45](2,0) to (3,0);
\draw[bend left=45](4,0) to (7, 0);
\draw[bend left=45](6,0) to (6.5, 0.5);
\node at (0, -1.0){$*8$};
\node at (1, -1.0){$4$};
\node at (2, -1.0){$1$};
\node at (3, -1.0){$1$};
\node at (4, -1.0){$3$};
\node at (5, -1.0){$4$};
\node at (6, -1.0){$*2$};
\node at (7, -1.0){$3$};

\end{tikzpicture}
\end{center}
Notice that any descendant of this open Skolem sequence will always  have $s_2=s_6=4$ and $s_3=s_4=1$. 
\end{example}


\section{A generating tree for open Skolem sequences}\label{sec:gentree}

We present a succession rule for the construction of all open Skolem sequences. To each open Skolem sequence $\rho \in \mathcal{OS}_n$, we associate two labels: 
\begin{enumerate}
\item The $\rho$ itself; 
\item A set $S$ made up of the unstarred elements $\rho$ (without multiplicity). 
\end{enumerate}

\begin{example} \label{ex:oss7b} Take $\rho \in \mathcal{OS}_7$ as seen in Example \ref{ex:oss7}. Then $\rho$ would have labels $(*7, 4,1,1,*3, 4,*1)$  and $S=\{1,4\}$. 
\end{example}
Skolem sequences may be recognized as the subset of open Skolem sequences of length $2n$ where $S= \{1,2, \ldots, n\}$. This may be tested using cardinalities and maximal elements: 
\begin{proposition} \label{prop:1} A Skolem sequence represented with labels of $\rho \in \mathcal{OS}_n $ and set $S$ satisfies the following: 
\begin{enumerate}
\item $2|S|=|n|$;
\item  $\max(S)= |S|$. 
\end{enumerate}
\end{proposition}
\begin{proof}
A Skolem sequence is made up of the elements of the multi-set $\{1,1, 2,2 \ldots, n,n\}$. When $k$ is present twice in $\rho \in \mathcal{OS}_n$, no $*$ decorates it, and $k$ is in the set $S$, thus Claim 1 is shown. Since each element of a Skolem sequence is present twice, the corresponding label set is $S=\{1,2, \ldots n\}$ whose cardinality is $n=\max(S)$, thus Claim 2. 
\end{proof}

We use these labels to generate all open Skolem sequences, and the above proposition to identify Skolem sequences. 

We use open Skolem sequences of order $n$ to build open Skolem sequences of order $n+1$. There are two ways of adding the next element: 
\begin{enumerate}
\item Adding an element with a $*$ decorating it at the end. 
\item Adding a copy of one of the $*$ elements to alter the $*$ element to a non-decorated state, provided that its value is not already present in the $S$ label set: 
\end{enumerate}
In the language of \cite{BuElMiYe2011}, (1) corresponds to the addition of an opener and (2) to the addition of a closer. 

\begin{example} In Examples \ref{ex:oss7} and \ref{ex:iss7} we saw an open Skolem sequence of order $7$ and two of its children. In Example \ref{ex:oss7b}, we established that the labels of the given $\rho$ would be: $(*7, 4,1,1,*3, 4,*1)$ and $\{1,4\}$. The labels of the first descendant seen in Example \ref{ex:iss7} is $(7,4,1,1,*4, 4, *2, 7)$ and $\{1,4,7\}$, while the labels of the second are $(*8, 4,1,1,3,4,*2, 3)$ and $\{1,3,4\}$.  
\end{example}

We describe the succession rule for constructing open Skolem sequences of order $n+1$ from an open Skolem sequence of order $n$:
\begin{enumerate}
\item \textbf{To add an opener:}  To each $*k\in \rho \in \mathcal{OS}_n \rightarrow *(k+1)$, and $*1$ is appended to $\rho$ to create a $\mu \in \mathcal{OS}_{n+1}$. 
\item  \textbf{To add a closer:} For each $*j \in \rho \in \mathcal{OS}_n$, check if $j \in S$. If yes, \textbf{stop}. If no:
\begin{itemize}
\item  For all $*k \neq *j$, $*k \rightarrow *(k+1)$, 
\item $*j \rightarrow j$,
\item  $j$ is appended to $\rho \in \mathcal{OS}_n$ to create a $\mu \in \mathcal{OS}_{n+1}$, and
\item $S\rightarrow  S \cup \{j\}$. 
\end{itemize}
\end{enumerate}

 We depict the start of the generating tree in Figure \ref{fig:onetree}.

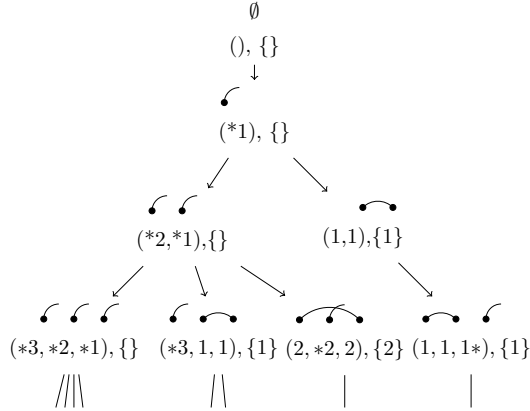
\begin{figure}[h!]  
 \centering
\begin{minipage}[b]{0.47\linewidth}
\scalebox{0.8}{
\begin{tikzpicture}[scale=0.6]
\node(1a) at (0,0){$\emptyset$};
\node(1label) at (0,-1){(), \{\}};
\node(2a) at (0, -3){\begin{tikzpicture}[scale=0.5]
\node[pnt] at (0,0){};
\draw[bend left=45](0,0) to (0.5, 0.5);
\node at (1,-1){(*1), \{\}};
\end{tikzpicture}};
\node(3a) at (-2, -6){\begin{tikzpicture}[scale=0.5]
\node[pnt] at (0,0){};
\node[pnt] at (1,0){};
\draw[bend left=45](0,0) to (0.5, 0.5);
\draw[bend left=45](1,0) to (1.5, 0.5);
\node at (1, -1){(*2,*1),\{\}};\end{tikzpicture}
};
\node(3b) at (3, -6){\begin{tikzpicture}[scale=0.5]
\node[pnt] at (0,0){};
\node[pnt] at (1,0){};
\draw[bend left=45](0,0) to (1,0);
\node at (0, -1){(1,1),\{1\}};\end{tikzpicture}
};
\node(4a) at (-5, -9){\begin{tikzpicture}[scale=0.5]
\node[pnt] at (0,0){};
\node[pnt] at (1,0){};
\node[pnt] at (2,0){};
\draw[bend left=45](0,0) to (0.5, 0.5);
\draw[bend left=45](1,0) to (1.5, 0.5);
\draw[bend left=45](2,0) to (2.5, 0.5);
\node at (1, -1){$(*3,*2,*1), \{\}$};\end{tikzpicture}
};
\node(4b) at (-1, -9){\begin{tikzpicture}[scale=0.5]
\node[pnt] at (0,0){};
\node[pnt] at (1,0){};
\node[pnt] at (2,0){};
\draw[bend left=45](0,0) to (0.5,0.5);
\draw[bend left=45](1,0) to (2, 0);
\node at (1.5, -1){$(*3,1,1), \{1\}$};\end{tikzpicture}
};
\node(4c) at (2.5, -9){\begin{tikzpicture}[scale=0.5]
\node[pnt] at (0,0){};
\node[pnt] at (1,0){};
\node[pnt] at (2,0){};
\draw[bend left=45](0,0) to (2,0);
\draw[bend left=45](1,0) to (1.5, 0.5);
\node at (1.5, -1){$(2, *2, 2), \{2\}$}; \end{tikzpicture}
};
\node(4d) at (6, -9){\begin{tikzpicture}[scale=0.5]
\node[pnt] at (0,0){};
\node[pnt] at (1,0){};
\node[pnt] at (2,0){};
\draw[bend left=45](0,0) to (1,0);
\draw[bend left=45](2,0) to (2.5, 0.5);
\node at (1.5, -1){$(1,1, 1*), \{1\}$};
\end{tikzpicture}
};

\draw[->](1label) to (2a);
\draw[->](2a) to (3a);
\draw[->](2a) to (3b);
\draw[->](3a) to (4a);
\draw[->](3a) to (4b);
\draw[->](3a) to (4c);
\draw[->](3b) to (4d);
\draw(4a) to (-5.5, -11);
\draw(4a) to (-5.25, -11);
\draw(4a) to (-5, -11);
\draw(4a) to (-4.75, -11);
\draw(4b) to (-0.8, -11);
\draw(4b) to (-1.2,-11);
\draw(4c) to (2.5, -11);
\draw(4d) to (6, -11); 
\end{tikzpicture}}
\end{minipage}

  \caption{The start of the generating tree.}
   \label{fig:onetree}
\end{figure}

\begin{example} We consider the open Skolem sequences that arise from adding a vertex to $\rho=(5*, 4*, 1,1,1*)\in \mathcal{OS}_5$ in Figure \ref{fig:32}.

\begin{figure}[h!]
\centering
\begin{minipage}[b]{5.2cm}
\scalebox{0.9}{
\begin{tikzpicture}[scale=0.8]
\node[pnt, label=below:1] at (0,0){};
\node[pnt,label=below:2] at (1,0){};
\node[pnt, label=below:3] at (2,0){};
\node[pnt,label=below:4] at (3,0){};
\node[pnt, label=below:5] at (4,0){};
\draw[bend left=45](0,0) to (0.5, 0.5);
\draw[bend left=45](1,0) to (1.5, 0.5);
\draw[bend left=45](2,0) to (3,0);
\draw[bend left=45](4,0) to (4.5, 0.5);

\node at (2, -1){$(*5, *4,1,1,*1), \{1\}$};
\end{tikzpicture}}
\end{minipage}
\begin{minipage}[b]{1cm}
\huge $\rightarrow$

\vspace{1.5cm}
\end{minipage}
\begin{minipage}[b]{5.5cm}
\scalebox{0.9}{
\begin{tikzpicture}[scale=0.8]
\node[pnt, label=below:1] at (0,0){};
\node[pnt,label=below:2] at (1,0){};
\node[pnt, label=below:3] at (2,0){};
\node[pnt,label=below:4] at (3,0){};
\node[pnt, label=below:5] at (4,0){};
\node[pnt, label=below:6] at (5,0){};
\draw[bend left=45](0,0) to (0.5, 0.5);
\draw[bend left=45](1,0) to (1.5, 0.5);
\draw[bend left=45](2,0) to (3,0);
\draw[bend left=45](4,0) to (4.5, 0.5);
\draw[bend left=45](5,0) to (5.5, 0.5);

\node at (2.5, -1){$(*6, *5, 1,1,*2, *1), \{1\}$};
\end{tikzpicture}}
\end{minipage}

\begin{minipage}[b]{5.5cm}
\scalebox{0.9}{
\begin{tikzpicture}[scale=0.8]
\node[pnt, label=below:1] at (0,0){};
\node[pnt,label=below:2] at (1,0){};
\node[pnt, label=below:3] at (2,0){};
\node[pnt,label=below:4] at (3,0){};
\node[pnt, label=below:5] at (4,0){};
\node[pnt, label=below:6] at (5,0){};
\draw[bend left=45](0,0) to (0.5, 0.5);
\draw[bend left=45](1,0) to (5, 0);
\draw[bend left=45](2,0) to (3,0);
\draw[bend left=45](4,0) to (4.5, 0.5);
\node at (2.5, -1){$(*6,4,1,1,*2, 4),\{1,4\}$};

\end{tikzpicture}}
\end{minipage} \hspace{0.5cm}
\begin{minipage}[b]{5.5cm}
\scalebox{0.9}{
\begin{tikzpicture}[scale=0.8]
\node[pnt, label=below:1] at (0,0){};
\node[pnt,label=below:2] at (1,0){};
\node[pnt, label=below:3] at (2,0){};
\node[pnt,label=below:4] at (3,0){};
\node[pnt, label=below:5] at (4,0){};
\node[pnt, label=below:6] at (5,0){};
\draw[bend left=45](0,0) to (5, 0);
\draw[bend left=45](1,0) to (1.5, 0.5);
\draw[bend left=45](2,0) to (3,0);
\draw[bend left=45](4,0) to (4.5, 0.5);
\node at (2.5, -1){$(5, *5, 1,1,*2,5),\{1,5\}$};

\end{tikzpicture}}
\end{minipage}
\caption{$w=(5*,4*,1,1,1*) \in \mathcal{OS}_5$ and its descendants.}
   \label{fig:32}
\end{figure}
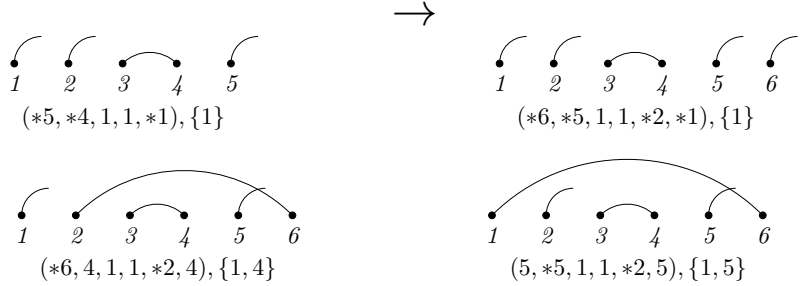
\end{example}



\section{Discussion}\label{sec:dis}

\noindent \textbf{Programming:} The succession rule has been successfully programmed, see \cite{BuWebsite} or Appendix 1, in Maple 16 using a 2 Intel Xeon. The generating tree produces open Skolem sequences which are then searched for Skolem sequences, using Proposition \ref{prop:1}. The total number of open Skolem sequences is:

\small

 $$1,2,4, 8, 20, 52, 146 , 430, 1306,4176,13832,47452 , 169044 , 619672, 234225$$

\normalsize

Note that the sequence made up of $|\mathcal{OS}_n|$ is not found in the Online Encycolpedia of Integer Sequences. Enumerating the number of open Skolem sequences requires the use of more memory than is given with our large computing power, which runs out after over 30GBs are used, and time is less than 10 minutes. Due to the nature of our labelling system insufficient memory to compute $\mathcal{OS}_{16}$ is unsurprising, if disappointing. That said, while we have not pushed the enumeration of Skolem sequences further, we have dramatically reduced the potential search space for future endeavors. For example, historically when exhaustively searching for all Skolem sequences of order $n=5$, all $(2\cdot 5)!=3,628,800$ permutations were considered. Using our method, only 4176 open Skolem sequences must be searched to find Skolem sequences of order $n=5$. We anticipate that there are less than 10 million open Skolem sequences to search in order to find the next exhaustive set of all 504 Skolem sequences of order 8, while classically there would be over $20 \times 10^{12}$ permutations to consider. 

Beyond this purely enumerative benefit of our method, it should be noted that through our generating tree, Skolem sequences are also exhaustively constructed, a definite goal in the field of design theory. Furthermore, this method lends itself to reasonable modifications that the authors are confident will lead to exhaustive generation and construction of other Skolem-type sequences. 






Indeed there are a variety of methods for extending Skolem sequences, including adding in \emph{hooks}, or 0's, in the sequence, and allowing more than one pair of entries to have value $k$. In the latter, this could include having entries with the same value $k$ all be $k$ units apart, or different pairs with the same value could be independent of each other, see \cite{ChThesis}.  In each case, with some careful bookkeeping, we are confident the succession rules may be carefully manipulated to include these extensions. 

While the enumeration of Skolem sequences, and various generalizations of Skolem sequences has been considered to varying degrees in various algorithms (\cite{ChThesis}, \cite{LaThesis}, \cite{GoWebsite}), the strategy presented here is a departure from them in two different ways. Firstly, Skolem sequences of size $n+1$ are built from open Skolem sequences of size $n$, and all may exhaustively be generated in this manner. Secondly, and perhaps more importantly, this method in which the sequence is represented as an arc diagram and then a label is used in order to produce a succession rule for a generating tree speaks to a  larger picture in which arc diagrams provide a unifying theory. By encoding Skolem sequences in arc diagrams, we may use this framework for construction and enumeration. This framework highlights the potential for many other combinatorial classes to be studied in this manner.

\textbf{Remark:} In \cite{ChDeDuStYa2007}, matchings are encoded in oscillating tableaux with integer fillings in order to prove equidistribution between crossing and nesting statistics. A Skolem sequence, when viewed as a subset of matchings, corresponds to those oscillating tableaux in which each integer filling appears $k$ times for every $k \in \{1,2, \ldots n\}$. 

\begin{example} We have seen the Skolem sequence $w=(3,4,2,3,2,4,1,1) \in \mathcal{S}_4$ drawn as an arc diagram: 

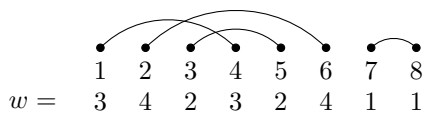
\begin{center}
\begin{tikzpicture}[scale=0.6]
\node[pnt, label=below:$1$] at (0,0){};
\node[pnt,label=below:$2$] at (1,0){};
\node[pnt, label=below:$3$] at (2,0){};
\node[pnt,label=below:$4$] at (3,0){};
\node[pnt, label=below:$5$] at (4,0){};
\node[pnt,label=below:$6$] at (5,0){};
\node[pnt, label=below:$7$] at (6,0){};
\node[pnt,label=below:$8$] at (7,0){};
\draw[bend left=45](0,0) to (3,0);
\draw[bend left=45](1,0) to (5,0);
\draw[bend left=45](2,0) to (4,0);
\draw[bend left=45](6,0) to (7,0);
\node at (-1.5, -1.2){$w=$};
\node at (0,-1.2){$3$};
\node at (1,-1.2){$4$};
\node at (2,-1.2){$2$};
\node at (3,-1.2){$3$};
\node at (4,-1.2){$2$};
\node at (5,-1.2){$4$};
\node at (6,-1.2){$1$};
\node at (7,-1.2){$1$};

\end{tikzpicture}
\end{center}
Its filled oscillating tableaux is as follows: 

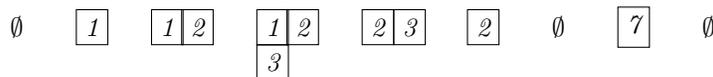
\begin{center}

\begin{tikzpicture}
\node at (0,0){$\emptyset$};
\node[rectangle, draw] at (1,0){1};
\node[rectangle, draw] at (2,0){1};
\node[rectangle, draw] at (2.4, 0){2};
\node[rectangle, draw] at (3.4, 0){1};
\node[rectangle, draw] at (3.8, 0){2};
\node[rectangle, draw] at (3.4, -0.47){3};
\node[rectangle, draw] at (4.8, 0){2};
\node[rectangle, draw] at (5.22, 0){3};
\node[rectangle, draw] at (6.2, 0){2};
\node[rectangle] at (7.2, 0){$\emptyset$};
\node[rectangle, draw] at (8.2, 0){7};
\node[rectangle] at (9.2, 0){$\emptyset$};
\end{tikzpicture}
\end{center}
We see that the 7 appears once, the 3 appears twice, the 1 appears three times, and the 2 four times. The presence of $n$ different integers, appearing each of the $\{1, \ldots, n\}$ times is unsurprising for a Skolem sequence. This is because an arc of length $k$ enters the tableaux at the presence of its closer, (when read right to left) and does not leave until its opener has been reached. 
\end{example}

 Future work includes the construction of generalized Skolem-type sequences using this method and the translation of the succession rule to a functional equation for faster enumeration.

\section*{Acknowledgements}
\label{sec:ack}
Thanks to Brett Stevens and Marni Mishna for directing us to this interesting problem. We are also very grateful to Mogens Lemvig Hansen for his invaluable help in Maple and Latex.


\bibliographystyle{plain}



\label{sec:biblio}

\appendix
\section{Maple code}

\begin{verbatim}
###  Preamble  ###########
`type/openSkolemlabel` := [ list(integer), set(posint) ];
incrementstar := proc(n::integer)
   if n >= 0 then 
      n;
   else
      n - 1;
   fi;
end;

addopener := proc(L::openSkolemlabel)
   [ [ op(map(incrementstar, L[1])), -1], L[2] ];
end;
addcloser := proc(L::openSkolemlabel)
   seq( addcl(L, i), i=1..nops(L[1]) );
end;
addcl := proc(L::openSkolemlabel, i::posint)
   local K;
   if L[1][i] >= 0 or member(-L[1][i], L[2]) then 
      return NULL;
   fi;
   K := L[1];
   K[i] := -K[i];
   [ [ op(map(incrementstar, K)), -L[1][i] ],
     L[2] union { -L[1][i] } ];
end;
Skolem := proc(L::openSkolemlabel)
   local n;
   n := nops(L[1])/2;
   type(n, integer)
      and
      L[2] = {$1..n};
end;
### end preamble ###############
L := [ [], {} ];
to 20 do
   L := op(map(addopener, [L])), op(map(addcloser, [L]));
   K := select(Skolem, [L]);
   n := nops(K);
   if n > 0 then
      print( n, op(K) );
   fi;
od:

\end{verbatim}
\end{document}